\documentclass[english]{amsart}

\usepackage{esint}
\usepackage[svgnames]{xcolor} 
\usepackage[colorlinks,citecolor=red,pagebackref,hypertexnames=false,breaklinks]{hyperref}
\usepackage{pgf,tikz}
\usepackage{pdfsync}

\usepackage{dsfont}
\usepackage{url}
\usepackage[utf8]{inputenc}
\usepackage[T1]{fontenc}
\usepackage{lmodern}
\usepackage{babel}
\usepackage{mathtools}  
\usepackage{amssymb}
\usepackage{lipsum}
\usepackage{mathrsfs}
\usepackage{color}
\usepackage[skip=2.2pt plus 1pt, indent=12pt]{parskip}

\newtheorem{theorem}{Theorem}[section]

\newtheorem{lemma}{Lemma}[section]

\newtheorem{corollary}{Corollary}[section]

\numberwithin{equation}{section}

\title[An $L^p$-$L^{p'}$ Carleman inequality for $p=\frac{2n}{n+2}$]{An $L^p$-$L^{p'}$ Carleman inequality for $p=\frac{2n}{n+2}$}

\author{Mourad Choulli}
\address{Universit\'{e} de Lorraine, 34 cours L\'{e}opold, 54052 Nancy cedex, France}
\email{mourad.choulli@univ-lorraine.fr}

\author{Hiroshi Takase}
\address{Department of Mathematics, Okayama University, 3-1-1 Tsushima-naka, Kita-ku, Okayama 700-8530, Japan}
\email{takase@math.okayama-u.ac.jp}

\thanks{This work was supported by JSPS KAKENHI Grant Numbers JP25K17280 and JP23KK0049.}

\date{}

\begin{document}
\begin{abstract}
We establish a new $L^p$-$L^{p'}$ type Carleman inequality for $p=\frac{2n}{n+2}$, which partially generalizes the now-classical Carleman inequality of Jerison and Kenig. Our result relies on a refinement of the proof of a Carleman inequality due to Dos Santos, Kenig, and Salo.
\end{abstract}



\maketitle

\section{Introduction}

Let $M'=(M',\mathfrak{g}')$ be a connected compact smooth Riemannian manifold  without boundary of dimension $n-1$, $n\ge 3$, and $M=\mathbb{R}\times M'$. We endow $M$ with the metric $\mathfrak{g}=e\oplus \mathfrak{g}'$, where $e$ denotes the Euclidean metric on $\mathbb{R}$. For all $\mu \ge 0$, consider  on $M$ the operator $Q^\mu$ given as follows:
\[
Q^\mu =\partial_t^2+2\mu \partial_t+\Delta'=(\partial_t+\mu)^2+\Delta'-\mu^2,
\]
where $\Delta'$ is the Laplace-Beltrami operator on $M'$. Note that when $\mu=0$, $Q^\mu$ is nothing but the Laplace-Beltrami operator, denoted by $\Delta$, on $M$.

Let $((\lambda_j^\mu)^2)$ be the sequence of eigenvalues of $-\Delta'+\mu^2$:  
\[
\mu=\lambda_0^\mu<\lambda_1^\mu\le \lambda_2^\mu\le \ldots\le \lambda_j^\mu\le \ldots,\quad \lambda_j^\mu\rightarrow \infty\; \mbox{as}\; j\rightarrow \infty. 
\]
Note that $(\lambda_j^\mu)^2 =(\lambda_j^0)^2+\mu^2$ ($j\ge 0$).

In what follows, we assume that the following gap condition holds: 
\begin{equation}\label{gap}
\kappa^\mu:=\inf\left\{\lambda_{j+1}^\mu-\lambda_j^\mu;\; j\ge 0,\; \lambda_{j+1}^\mu\ne \lambda_j^\mu\right\}>0.
\end{equation}

Let us provide an example, which will be used later, in which \eqref{gap} is satisfied. In the case $M'=\mathbb{S}^{n-1}$, the unit sphere of $\mathbb{R}^n$ equipped with its canonical metric, we verify that 
\[
\lambda_k^{\frac{n-2}{2}}=k+\frac{n-2}{2},\quad k\ge 1.
\]
Therefore, \eqref{gap} holds with $\kappa^{\frac{n-2}{2}}=1$.

We also emphasize that \eqref{gap} remains valid under a small perturbation of the metric $\mathfrak{g'}$. Indeed, one can define a complete distance $\varrho$ on $\mathscr{M}'$--the set of $C^\infty$ metrics on $M'$--such that the following result holds: for all $\delta >0$ and $\mathfrak{g}'_1, \mathfrak{g}'_2\in \mathscr{M}'$, the inequality $\varrho(\mathfrak{g}'_1, \mathfrak{g}'_2)<\delta$ implies
\[
e^{-(n+1)\delta}\le \frac{(\lambda_{k,1})^2}{(\lambda_{k,2})^2}\le e^{(n+1)\delta},\quad k\ge 1,
\]
where $((\lambda_{k,j})^2)$ is the sequence of the eigenvalues of $-\Delta'$ when $\mathfrak{g}'=\mathfrak{g}'_j$ ($j=1,2$). We refer to \cite[Theorem 2.2]{SU} for further details.

Let $p=\frac{2n}{n+2}$ and $p'=\frac{2n}{n-2}$. Note that $p$ and $p'$ are conjugate, that is, we have $\frac{1}{p}+\frac{1}{p'}=1$. Fix $0<\varsigma < \frac{\kappa^\mu}{2}$, let $\tau(n)=\max\left(\frac{4(n-1)}{n-2},n+1\right)\ge 6$ and define
\[
\Lambda^\mu =\{\tau\in \mathbb{R};\; |\tau| \in [\tau(n),\infty); \; \mathrm{dist}(|\tau|,\{\lambda_j^\mu ;\; j\ge 0\})\ge \varsigma\}.
\]

We aim to establish the following Carleman inequality. From now on, $\mathbf{c}=\mathbf{c}(n,M')>0$ will denote a generic constant.

\begin{theorem}\label{thmci}
For all $\tau \in \Lambda^\mu$ and $u\in C_0^\infty (M)$, we have
\begin{equation}\label{ci1}
\|u\|_{L^{p'}(M)}\le \mathbf{c}_\mu \varsigma^{-\frac{2}{p'}}\|(e^{(\tau+\mu) t}Q^\mu e^{-(\tau+\mu) t})u\|_{L^p(M)},
\end{equation}
where $\mathbf{c}_0=\mathbf{c}$ and $\mathbf{c}_\mu=\mathbf{c}\mu^{-1}$ if $\mu > 0$.
\end{theorem}

By simply using Hölder's inequality, we obtain the following corollary as a consequence of Theorem \ref{thmci}.

\begin{corollary}\label{corci}
There exists a constant $\tilde{\mathbf{c}}_\mu=\tilde{\mathbf{c}}_\mu(n,M',\mu)>0$ such that $\tilde{\mathbf{c}}_0=\tilde{\mathbf{c}}_0(n,M')$ and $\tilde{\mathbf{c}}_\mu=\tilde{\mathbf{c}}(n,M')\mu^{-1}>0$ if $\mu > 0$, and   for all $V\in L^{\frac{n}{2}}(M)$ satisfying $\|V\|_{L^{\frac{n}{2}}(M)}\le \tilde{\mathbf{c}}_\mu\varsigma^{\frac{2}{p'}}$, $\tau \in \Lambda^\mu$ and $u\in C_0^\infty (M)$ we have
\[
\|u\|_{L^{p'}(M)}\le \mathbf{c}_\mu\varsigma^{-\frac{2}{p'}}\|(e^{(\tau+\mu) t}(Q^\mu+V) e^{-(\tau+\mu) t})u\|_{L^p(M)},
\]
where $\mathbf{c}_\mu$ is a constant of the form given in Theorem \ref{thmci}.
 \end{corollary}

Let $I$ be a compact interval of $\mathbb{R}$. It is proved in \cite[Theorem 1.3]{DKS} that the following Carleman inequality holds without the gap condition \eqref{gap} : there exists a constant $\mathfrak{c}=\mathfrak{c}(n,M')>0$ such that
\begin{equation}\label{DKS}
\|u\|_{L^{p'}(I \times M')}\le \mathfrak{c}(1+|I|^{\frac{1}{p'}})\|(e^{\tau t}\Delta e^{-\tau t})u\|_{L^p(I \times M')},
\end{equation}
for all $u\in C_0^\infty (I\times M')$ and $|\tau|\in [4,\infty)\setminus \{\lambda_j^0;\; j\ge 1\}$.

It is interesting to note that, by continuity, \eqref{DKS} remains valid for all $|\tau| \in [4, \infty)$.

Our proof of Theorem \ref{thmci} is obtained by refining the proof of \eqref{DKS}. When $\mu=0$, \eqref{ci1} corresponds to a variant of \eqref{DKS} in which the dependence of the constant on the interval $I$ is replaced by the gap condition \eqref{gap}. Furthermore, \eqref{ci1} generalizes partially the $L^p$-$L^{p'}$ Carleman inequality of Jerison and Kenig (see \cite[Theorem 2.1]{JK}), a now-classic result. To see this, assume that $M'=\mathbb{S}^{n-1}$ is equipped with its canonical metric and set
\[
\dot{C}_0^\infty(\mathbb{R}^n):=\{f\in C_0^\infty (\mathbb{R}^n);\; \mathrm{supp}(f)\subset \mathbb{R}^n\setminus\{0\}\}.
\]
Take $u\in \dot{C}_0^\infty(\mathbb{R}^n)$ and define $w\in C^\infty_0(\mathbb{R}\times \mathbb{S}^{n-1})$ by $w(t,\theta)=u(e^t\theta)$, where $(t,\theta)\in \mathbb{R}\times \mathbb{S}^{n-1}$. It can then be verified that $e^{2t}\Delta u(e^t\theta)=Q^{\frac{n-2}{2}}w(t,\theta)$. As noted above, the gap condition \eqref{gap} is satisfied in this case, with $\kappa^{\frac{n-2}{2}}=1$.

In the following $c=c(n)>0$ is a generic constant.

Let $\tau \in \Lambda^{\frac{n-2}{2}}$. By applying \eqref{ci1} to $e^{(\tau+\frac{n-2}{2})t}w$, we obtain
\[
\|e^{(\tau+\frac{n-2}{2})t}w\|_{L^{p'}(\mathbb{R} \times \mathbb{S}^{n-1})}\le c\varsigma^{-\frac{2}{p'}}\|(e^{(\tau+\frac{n-2}{2}) t}Q^\frac{n-2}{2} w\|_{L^p(\mathbb{R} \times \mathbb{S}^{n-1})}.
\]
Then, a change of variable yields
\[
\||x|^{\tau+\frac{n-2}{2}}u\|_{L^{p'}(\mathbb{R}^n,|x|^{-n}dx)}\le c\varsigma^{-\frac{2}{p'}}\||x|^{\tau+\frac{n-2}{2}+2}\Delta u\|_{L^p(\mathbb{R}^n,|x|^{-n}dx)}.
\]
Suppose that $\tau>0$. By replacing $\tau$ with $-\tau$ in this inequality, and subsequently substituting $\tau - \frac{n-2}{2}$ with $\tau$, we obtain, for $\tau \in \Lambda_\varsigma := \{\lambda \in [\tau_0(n), \infty[ : \mathrm{dist}(\lambda, \mathbb{N}) > \varsigma\}$, where $\tau_0(n) = \tau(n) - \frac{n-2}{2}$ ($\ge 9/2$),
\[
\||x|^{-\tau}u\|_{L^{p'}(\mathbb{R}^n,|x|^{-n}dx)}\le c\varsigma^{-\frac{2}{p'}}\||x|^{-\tau+2}\Delta u\|_{L^p(\mathbb{R}^n,|x|^{-n}dx)}.
\]
This inequality corresponds to \cite[Theorem 2.1 with $\mu=1$]{JK}. Unlike the inequality in \cite[Theorem 2.1]{JK}, the constant appearing in the inequality above depends explicitly on $\varsigma$. A direct proof of the above inequality will be provided in \cite{Ch}.

\section{Proof of Theorem \ref{thmci}}

In this section, $\mu \ge 0$ is fixed, and for the sake of simplicity, we omit the superscript $\mu$ from all the quantities we will use.
Let $(\phi_j)$ be a sequence of eigenfunctions, corresponding to the eigenvalues $(\lambda_j ^2)$,  chosen such that $(\phi_j)$ forms an orthonormal basis of $L^2(M')$. Let 
\[
p_j:L^2(M')\rightarrow L^2(M'): f\mapsto (f|\phi_j)\phi_j,
\]
where $(\cdot|\cdot)$ stands for the usual inner product on $L^2(M')$. Define the spectral clusters as follows
\[
\pi_\lambda=\sum_{\lambda\le \lambda_j<\lambda+1}p_j,\quad \lambda \ge 0.
\]
For further use, note that $\pi_\lambda$ is self-adjoint and $\pi_\lambda^2=\pi_\lambda$, for all $\lambda\ge 0$.

\begin{lemma}\label{lem1}
For all $\lambda\ge 0$ and $f\in L^2(M')$ we have
\begin{align}
&\|\pi_\lambda f\|_{L^{p'}(M')}\le \mathbf{c}(1+\lambda)^{\frac{1}{p'}}\|f\|_{L^2(M')}  ,\label{up4}
\\
&\|\pi_\lambda f\|_{L^2(M')}\le \mathbf{c}(1+\lambda)^{\frac{1}{p'}}\|f\|_{L^p(M')}.\label{up5}
\end{align}
\end{lemma}

\begin{proof}
It follows from \cite[Corollary 5.1.2]{Sogge}, in which $n$ is replaced by $n-1$, 
\[
\|\pi_\lambda f\|_{L^{p'}(M')}\le \mathbf{c}(1+\lambda)^{\frac{1}{p'}}\|f\|_{L^2(M')} .
\]
This is \eqref{up4}. On the other hand, since 
\[
\|\pi_\lambda f\|_{L^2(M')}^2=(\pi_\lambda f|\pi_\lambda f)=(\pi_\lambda^2f|f)=(\pi_\lambda f|f), 
\]
we get by applying H\"older's inequality
\[
\|\pi_\lambda f\|_{L^2(M')}^2\le \|\pi_\lambda f\|_{L^{p'}(M')}\|f\|_{L^p(M')}.
\]
Combining this inequality with \eqref{up4} applied to $\pi_\lambda f$, we obtain \eqref{up5}.
\end{proof}

\begin{proof}[Proof of Theorem \ref{thmci}]

Let $\tau \in \Lambda$, $u\in C_0^\infty(M)$ and 
\[
f:=e^{(\tau+\mu) t} Qe^{-(\tau+\mu) t}u=(\partial_t^2-2\tau \partial_t +\tau^2-\mu^2 +\Delta')u.
\]

In the following,  the Fourier transform with respect to $t$ is denoted by $\mathcal{F}$. Since $p_j((\Delta' -\mu^2)u)=-\lambda_j^2 p_j u$, we obtain
\begin{align*}
\mathcal{F}(p_jf(\cdot,x'))(\xi)&= ((i\xi)^2-2i\tau \xi +\tau^2-\lambda_j^2)\mathcal{F}(p_ju(\cdot,x'))(\xi)
\\
&=((i\xi-\tau)^2-\lambda_j^2)\mathcal{F}(p_ju(\cdot,x'))(\xi)
\\
&=(i\xi-\tau-\lambda_j)(i\xi-\tau+\lambda_j)\mathcal{F}(p_ju(\cdot,x'))(\xi).
\end{align*}

In the remaining part of this proof, $t\in \mathbb{R}$ and $x'\in M'$. Since $\tau \ne \lambda_j$ for all $j\ge 0$, we obtain
\begin{align*}
p_ju (t,x')&=\frac{1}{2\pi}  \int_\mathbb{R}\frac{e^{it\xi}}{(i\xi-\tau-\lambda_j)(i\xi-\tau+\lambda_j)}\mathcal{F}(p_jf(\cdot,x'))(\xi)d\xi
\\
&= \frac{1}{2\pi}  \int_{\mathbb{R}}\int_\mathbb{R}\frac{e^{i(t-s)\xi}}{(i\xi-\tau-\lambda_j)(i\xi-\tau+\lambda_j)}p_jf(s,x'))d\xi ds.
\end{align*}
Hence,
\[
p_ju (t,x')=\int_{\mathbb{R}}m_j^\tau(t-s)p_jf(s,x')ds,
\]
where
\[
m_j^\tau(\eta):=\frac{1}{2\pi}  \int_\mathbb{R}\frac{e^{i\eta\xi}}{(i\xi-\tau-\lambda_j)(i\xi-\tau+\lambda_j)}d\xi.
\]
Using
\[
\|u(t,\cdot)\|_{L^{p'}(M')}=\left\|\sum_{k\ge 0}\pi_k^2u(t,\cdot)\right\|_{L^{p'}(M')}\le \sum_{k\ge 0}\|\pi_k^2u(t,\cdot)\|_{L^{p'}(M')},
\]
and \eqref{up4}, we obtain
\begin{equation}\label{eq1}
\|u(t,\cdot)\|_{L^{p'}(M')}\le \mathbf{c}\sum_{k\ge 0}(1+k)^{\frac{1}{p'}} \|\pi_k u(t,\cdot)\|_{L^2(M')}.
\end{equation}

On the other hand, we have
\begin{align*}
\|\pi_k u(t,\cdot)\|_{L^2(M')}^2&=\sum_{k\le \lambda_j<k+1}|(u(t,\cdot)|\phi_j)|^2
\\
&=\sum_{k\le \lambda_j<k+1}\left|\int_{\mathbb{R}} m_j^\tau(t-s)(f(s,\cdot)|\phi_j)ds \right|^2.
\end{align*}
Applying Minkowski's inequality, we obtain
\begin{align*}
\|\pi_k u(t,\cdot)\|_{L^2(M')}&\le \int_{\mathbb{R}}\left(\sum_{k\le \lambda_j<k+1}|m_j^\tau(t-s)(f(s,\cdot)|\phi_j)|^2\right)^{\frac{1}{2}}ds
\\
&\le \int_{\mathbb{R}}\max_{k\le \lambda_j<k+1}|m_j^\tau(t-s)|\left(\sum_{k\le \lambda_j<k+1}|(f(s,\cdot)|\phi_j)|^2\right)^{\frac{1}{2}}ds
\\
&\le \int_{\mathbb{R}}\max_{k\le \lambda_j<k+1}|m_j^\tau(t-s)|\|\pi_kf(\cdot,s)\|_{L^2(M')}ds,
\end{align*}
which, combined with \eqref{up5}, gives
\[
\|\pi_k u(t,\cdot)\|_{L^2(M')}\le \mathbf{c}\int_{\mathbb{R}}(1+k)^{\frac{1}{p'}}\max_{k\le \lambda_j<k+1}|m_j^\tau(t-s)|\|f(\cdot,s)\|_{L^p(M')}ds.
\]
This  in \eqref{eq1} yields
\[
\|u(t,\cdot)\|_{L^{p'}(M')}\le \mathbf{c}\sum_{k\ge 0}(1+k)^{\frac{2}{p'}} \int_{\mathbb{R}}\max_{k\le \lambda_j<k+1}|m_j^\tau(t-s)|\|f(\cdot,s)\|_{L^{p}(M')}ds,
\]
which we rewrite as
\begin{equation}\label{eq2}
\|u(t,\cdot)\|_{L^{p'}(M')}\le \mathbf{c}\sum_{k\ge 0}A_k(t),
\end{equation}
where
\[
A_k(t):=(1+k)^{\frac{2}{p'}} \int_{\mathbb{R}}\max_{k\le \lambda_j<k+1}|m_j^\tau(t-s)|\|f(\cdot,s)\|_{L^p(M')}ds,\quad k\ge 0.
\]

It follows from \eqref{ap3} and \eqref{ap4} in Appendix \ref{appA} that 
\begin{equation}\label{m}
|m_j^\tau(t-s)|\le 
\left\{
\begin{array}{ll}
\frac{e^{-|t-s|||\tau|-\lambda_j|}}{\lambda_j},\quad &\mu\ne 0\; \mbox{or}\; j\ge 1,
\\
\\
|t-s|e^{-|t-s||\tau|}, &\mu=0\; \mbox{and}\; j=0. 
\end{array}
\right.
\end{equation}
Given this inequality, we observe that we can restrict the remainder of the proof to the case $\tau>0$.

For $1\le k\le \lfloor \tau\rfloor-2$, we have
\[
\max_{k\le \lambda_j<k+1}|m_j^\tau(t-s)|\le \frac{1}{k}\max_{k\le \lambda_j<k+1}e^{-(\tau -\lambda_j)|t-s|}\le \frac{e^{-(\tau-1-k)|t-s|}}{k}.
\]
Hence,
\begin{align*}
\sum_{k=2}^{\lfloor \tau\rfloor-2}A_k(t) &\le  \sum_{k=2}^{\lfloor \tau\rfloor-2} (1+k)^{\frac{2}{p'}}k^{-1}\int_{\mathbb{R}}e^{-(\tau-1-k)|t-s|}\|f(\cdot,s)\|_{L^p(M')}ds
\\
&\le 2^{\frac{2}{p'}}\sum_{k=2}^{\lfloor \tau\rfloor-2} k^{\frac{2}{p'}-1}\int_{\mathbb{R}}e^{-(\tau-1-k)|t-s|}\|f(\cdot,s)\|_{L^p(M')}ds.
\end{align*}

Define
\[
g(\rho)=\rho^{\frac{2}{p'}-1}e^{-(\tau-1-\rho)|t-s|},\quad \rho\in [1,\lfloor \tau\rfloor-2].
\]
Then
\begin{align*}
g'(\rho)&=\left[\left(\frac{2}{p'}-1\right)+|t-s|\rho \right]\rho^{\frac{2}{p'}-2}e^{-(\tau-1-\rho)|t-s|}
\\
&=\left[-\alpha+|t-s|\rho \right]\rho^{\frac{2}{p'}-2}e^{-(\tau-1-\rho)|t-s|},
\end{align*}
where
\[
\alpha:=1-\frac{2}{p'}\in (0,1).
\]

Assume first that $|t-s|\le \frac{\alpha}{\lfloor \tau\rfloor-2}:=\gamma$, for which $g$ is non increasing on $[1,\lfloor \tau\rfloor-2]$. Thus,
\begin{align*}
&\sum_{k=2}^{\lfloor \tau\rfloor-2} k^{\frac{2}{p'}-1}\int_{|t-s|\le \gamma}e^{-(\tau-1-k)|t-s|}\|f(\cdot,s)\|_{L^p(M')}ds
\\
&\hskip 3cm \le \mathbf{c} \int_{|t-s|\le \gamma}I_\tau(|t-s|)\|f(\cdot,s)\|_{L^p(M')}ds,
\end{align*}
where
\[
I_\tau(|t-s|)=\int_1^{\lfloor \tau\rfloor-2}\rho^{\frac{2}{p'}-1}e^{-(\tau-1-\rho)|t-s|}d\rho.
\]
We have
\begin{align*}
I_\tau(|t-s|) &=|t-s|^{-\frac{2}{p'}}\int_{|t-s|}^{(\lfloor \tau\rfloor-2)|t-s|}r^{\frac{2}{p'}-1}e^{-[(\tau-1)|t-s|-r]}dr
\\
&\le |t-s|^{-\frac{2}{p'}}\int_0^\alpha r^{\frac{2}{p'}-1}dr
\\
&\le \mathbf{c} |t-s|^{-\frac{2}{p'}}.
\end{align*}
Therefore, we obtain
\begin{align}
&\sum_{k=2}^{\lfloor \tau\rfloor-2} k^{\frac{2}{p'}-1}\int_{|t-s|\le \gamma}e^{-(\tau-1-k)|t-s|}\|f(\cdot,s)\|_{L^p(M')}ds\label{eq3}
\\
&\hskip 2cm \le \mathbf{c}\int_{|t-s|\le \gamma}|t-s|^{-\frac{2}{p'}}\|f(\cdot,s)\|_{L^p(M')}ds.\nonumber
\end{align}

Consider now the case $|t-s|>\gamma$. Using $\sup_{\eta >0}\eta^{\frac{2}{p'}}e^{-\eta}<\infty$,
we obtain
\[
g(\rho)\le \mathbf{c}|t-s|^{-\frac{2}{p'}}\rho^{-1+\frac{2}{p'}}(\lfloor \tau \rfloor-1-\rho)^{-\frac{2}{p'}},\quad \rho \in [1,\lfloor \tau \rfloor-2],
\]
which we rewrite as
\[
g(\rho)\le \mathbf{c}|t-s|^{-\frac{2}{p'}}\sigma (\rho),\quad \rho \in [1,\lfloor \tau \rfloor-2],
\]
where
\[
\sigma (\rho)=\rho^{-1+\frac{2}{p'}}(\lfloor \tau \rfloor-1-\rho)^{-\frac{2}{p'}},\quad \rho \in [1,\lfloor \tau \rfloor-2].
\]

In what follows, we use that the assumption $\tau \ge \tau(n)$ implies $2\le \lfloor \alpha(\lfloor \tau \rfloor-1)\rfloor\le \lfloor \tau \rfloor-3$. Since
\[
\sigma'(\rho)=\left[ -\alpha(\lfloor \tau \rfloor-1)+\rho\right]\rho^{-2+\frac{2}{p'}}(\lfloor \tau \rfloor-1-\rho)^{-1-\frac{2}{p'}},\quad \rho \in [1,\lfloor \tau \rfloor-2],
\]
$\sigma$ is non increasing on $[1,\alpha(\lfloor \tau \rfloor-1)]$.  In consequence, we obtain
\begin{align*}
\sum_{k=2}^{\lfloor \alpha(\lfloor \tau \rfloor-1)\rfloor} g(k)&\le \mathbf{c}|t-s|^{-\frac{2}{p'}}\int_1^{\alpha(\lfloor \tau \rfloor-1)}\rho^{-1+\frac{2}{p'}}(\lfloor \tau \rfloor-1-\rho)^{-\frac{2}{p'}}d\rho
\\
&= \mathbf{c}|t-s|^{-\frac{2}{p'}}\int_{(\lfloor \tau \rfloor-1)^{-1}}^{\alpha}\rho^{-1+\frac{2}{p'}}(1-\rho)^{-\frac{2}{p'}}d\rho
\\
&\le \mathbf{c}|t-s|^{-\frac{2}{p'}}\int_0^{\alpha}\rho^{-1+\frac{2}{p'}}(1-\rho)^{-\frac{2}{p'}}d\rho.
\end{align*}
That is we have
\begin{equation}\label{eq3.1}
\sum_{k=2}^{\lfloor \alpha(\lfloor \tau \rfloor-1)\rfloor} g(k)\le \mathbf{c}|t-s|^{-\frac{2}{p'}}.
\end{equation}

On the other hand, we have
\begin{align*}
&\sum_{k=\lfloor \alpha(\lfloor \tau \rfloor-1)\rfloor+1}^{\lfloor \tau \rfloor-2} \gamma^{-1+\frac{2}{p'}} |t-s|g(k)
\\
&\hskip 2cm=\sum_{k=\lfloor \alpha(\lfloor \tau \rfloor-1)\rfloor+1}^{\lfloor \tau \rfloor-2} \gamma^{-1+\frac{2}{p'}}|t-s|k^{-1+\frac{2}{p'}}e^{-(\tau-1-k)|t-s|}
\\
&\hskip 2cm\le\sum_{k=\lfloor \alpha(\lfloor \tau \rfloor-1)\rfloor+1}^{\lfloor \tau \rfloor-2} \gamma^{-1+\frac{2}{p'}}|t-s|[\alpha(\lfloor \tau \rfloor-2)]^{-1+\frac{2}{p'}}e^{-(\tau-1-k)|t-s|}
\\
&\hskip 2cm\le \mathbf{c}\sum_{k=\lfloor \alpha(\lfloor \tau \rfloor-1)\rfloor+1}^{\lfloor \tau \rfloor-2} |t-s|e^{-(\tau-1-k)|t-s|}
\\
&\hskip 2cm\le \mathbf{c} \frac{|t-s|}{e^{|t-s|}-1}.
\end{align*}
As $\sup_{\eta >0} \frac{\eta}{e^{\eta}-1}<\infty$, we get
\[
\sum_{k=\lfloor \alpha(\lfloor \tau \rfloor-1)\rfloor+1}^{\lfloor \tau \rfloor-2} \gamma^{-1+\frac{2}{p'}} |t-s|g(k)\le \mathbf{c},
\]
But
\[
\gamma^{-1+\frac{2}{p'}} |t-s|\ge |t-s|^{\frac{2}{p'}}.
\]
Therefore, we obtain
\begin{equation}\label{eq3.2}
\sum_{\lfloor \alpha(\lfloor \tau \rfloor-1)\rfloor+1}^{\lfloor \tau \rfloor-2} g(k)\le \mathbf{c}|t-s|^{-\frac{2}{p'}}.
\end{equation}
Putting together \eqref{eq3.1} and \eqref{eq3.2}, we get
\[
\sum_{k=2}^{\lfloor \tau \rfloor-2} g(k)\le \mathbf{c}|t-s|^{-\frac{2}{p'}},
\]
from which we obtain
\begin{align}
&\sum_{k=2}^{\lfloor \tau\rfloor-2} k^{\frac{2}{p'}-1}\int_{|t-s|> \gamma}e^{-(\tau-1-k)|t-s|}\|f(\cdot,s)\|_{L^p(M')}ds\label{eq4}
\\
&\hskip 3cm\le \mathbf{c}\int_{|t-s|> \gamma}|t-s|^{-\frac{2}{p'}}\|f(\cdot,s)\|_{L^p(M')}ds.\nonumber
\end{align}
A combination of \eqref{eq3} and \eqref{eq4} gives
\begin{align*}
&\sum_{k=2}^{\lfloor \tau\rfloor-2} k^{\frac{2}{p'}-1}\int_{\mathbb{R}}e^{-(\tau-1-k)|t-s|}\|f(\cdot,s)\|_{L^p(M')}ds\
\\
&\hskip 3cm \le \mathbf{c}\int_{\mathbb{R}}|t-s|^{-\frac{2}{p'}}\|f(\cdot,s)\|_{L^p(M')}ds.
\end{align*}
Therefore, we have
\begin{equation}\label{eq5}
\sum_{k=2}^{\lfloor \tau\rfloor-2} A_k(t)
 \le \mathbf{c}\int_{\mathbb{R}}|t-s|^{-\frac{2}{p'}}\|f(\cdot,s)\|_{L^p(M')}ds.
\end{equation}

Now, we discuss the case $k\ge \lfloor \tau\rfloor+2$, for which we have
\[
\max_{k\le \lambda_j<k+1}|m_j^\tau(t-s)|\le \frac{1}{k}\max_{k\le \lambda_j<k+1}e^{-(\lambda_j-\tau)|t-s|}\le\frac{e^{-(k-\tau)|t-s|}}{k}.
\]
Define
\[
h(\rho)=\rho^{\frac{2}{p'}-1}e^{-(\rho-\tau)|t-s|},\quad \rho\ge \lfloor \tau\rfloor+2.
\]
Since
\[
h'(\rho)=\left[-\alpha -|t-s|\rho\right] \rho^{\frac{2}{p'}-2} e^{-(\rho-\tau)|t-s|}\le 0,\quad \rho\ge \lfloor \tau\rfloor+1,
\]
we get
\begin{equation}\label{eq6}
\sum_{k\ge  \lfloor \tau\rfloor+2}A_k(t)
  \le \mathbf{c} \int_{\mathbb{R}}J_\tau(|t-s|)\|f(\cdot,s)\|_{L^p(M')}ds,
\end{equation}
where
\[
J_\tau(|t-s|)=\int_{ \lfloor \tau\rfloor+1}^\infty \rho^{\frac{2}{p'}-1}e^{-(\rho-\tau)|t-s|}d\rho.
\]
If $(\lfloor \tau\rfloor+1)|t-s|<1$, then 
\begin{align*}
J_\tau(|t-s|)&\le |t-s|^{-\frac{2}{p'}}\int_{ (\lfloor \tau\rfloor+1)|t-s|}^\infty r^{\frac{2}{p'}-1}e^{-(r-\tau |t-s|)}dr
\\
&\le |t-s|^{-\frac{2}{p'}}\left(\int_0^1r^{\frac{2}{p'}-1}d\rho+\int_1^\infty e^{-r+1}dr\right)
\\
&\le \mathbf{c} |t-s|^{-\frac{2}{p'}}.
\end{align*}
When $(\lfloor \tau\rfloor+1)|t-s|\ge 1$, we have
\begin{align*}
J_\tau(|t-s|)&\le |t-s|^{-\frac{2}{p'}}\int_{ (\lfloor \tau\rfloor+1)|t-s|}^\infty e^{-(r-\tau |t-s|)}dr
\\
&=  |t-s|^{-\frac{2}{p'}} \int_{(\lfloor \tau\rfloor+1-\tau)|t-s|}^\infty e^{-r}dr
\\
&\le  |t-s|^{-\frac{2}{p'}}\int_0^\infty e^{-r}dr
\\
&\le \mathbf{c} |t-s|^{-\frac{2}{p'}}.
\end{align*}
The preceding inequalities in \eqref{eq6} yield
\begin{equation}\label{eq7}
\sum_{k\ge  \lfloor \tau\rfloor+2}A_k(t) 
 \le \mathbf{c} \int_{\mathbb{R}}|t-s|^{-\frac{2}{p'}}\|f(\cdot,s)\|_{L^p(M')}ds.
\end{equation}

When $k=0$, we have from \eqref{m}
\[
\max_{0\le \lambda_j<1}|m_j^\tau (t-s)|\le 
\left\{
\begin{array}{lll}
|t-s|e^{-\tau |t-s|}, \quad \mu =0\; \mbox{and}\; \{0\le \lambda_j<1\}=\{0\},
\\
\\
\max\left(|t-s|e^{-\tau |t-s|}, \frac{e^{-\frac{\tau}{2}[t-s|}}{\lambda_1}\right)\; \quad \mu =0\; \mbox{and}\; \lambda_1<1,
\\
\\
\frac{e^{-\frac{\tau}{2}[t-s|}}{\mu} \quad \mu \ne 0.
\end{array}
\right.
\]
As $\displaystyle \sup_{\eta >0,i=0,1}\eta^{\frac{2}{p'}+i}e^{-2\eta}<\infty$, we obtain
\begin{equation}\label{eq11}
A_0(t)\le \mathbf{c}_\mu \int_{\mathbb{R}}|t-s|^{-\frac{2}{p'}}\|f(\cdot,s)\|_{L^p(M')}ds,
\end{equation}
where $\mathbf{c}_0=\mathbf{c}$ and $\mathbf{c}_\mu=\mathbf{c}\mu^{-1}$ if $\mu \ne 0$.

While for $k=1$, since
\[
\max_{1\le \lambda_j<2}|m_j^\tau (t-s)|\le e^{-(\tau-2)|t-s|}\le e^{-2|t-s|} ,
\]
proceeding similarly as for $A_0(t)$, we get
\begin{equation}\label{eq11.1}
A_1(t)\le \int_{\mathbb{R}}|t-s|^{-\frac{2}{p'}}\|f(\cdot,s)\|_{L^p(M')}ds.
\end{equation}

Let us now consider the remaining case $k\in \{\lfloor \tau\rfloor-1,\lfloor \tau\rfloor, \lfloor \tau\rfloor+1\}$, for which we have
\[
\max_{k\le \lambda_j<k+1}|m_j^\tau(t-s)|\le \frac{1}{k}e^{-\varsigma |t-s|}\le e^{-\varsigma|t-s|}.
\]
As $\sup_{\eta >0}\eta^{\frac{2}{p'}}e^{-\varsigma \eta}=\varsigma^{-\frac{2}{p'}}\sup_{\eta >0}\eta^{\frac{2}{p'}}e^{-\eta}$, we proceed once again as for  $A_0$ to obtain
\begin{equation}\label{eq12}
A_k(t)\le \mathbf{c}\varsigma^{-\frac{2}{p'}}\int_{\mathbb{R}}|t-s|^{-\frac{2}{p'}}\|f(\cdot,s)\|_{L^p(M')}ds.
\end{equation}

Putting together \eqref{eq2}, \eqref{eq5}, \eqref{eq7}, \eqref{eq11}, \eqref{eq11.1} and \eqref{eq12}, we end up getting
\begin{equation}\label{eq8}
\|u(t,\cdot)\|_{L^{p'}(M')}\le \mathbf{c}\varsigma^{-\frac{2}{p'}} \int_{\mathbb{R}}|t-s|^{-\frac{2}{p'}}\|f(\cdot,s)\|_{L^p(M')}ds.
\end{equation}

Finally, applying Hardy-Littlewood-Sobolev's inequality to the right hand side of \eqref{eq8}, we get
\[
\|u\|_{L^{p'}(M)}\le \mathbf{c}\varsigma^{-\frac{2}{p'}}\|f\|_{L^p(M)}.
\]
This is the expected inequality.
\end{proof}

\appendix
\section{Auxiliary inequalities}
\label{appA}

Let $a,b\in \mathbb{R}\setminus\{0\}$ with $a\ne b$. When $\eta >0$, applying the residue theorem, we obtain
\[
\int_{\mathbb{R}}\frac{e^{i\eta \xi}}{(i\xi+a)(i\xi+b)}d\xi=
\left\{
\begin{array}{lll} 
0,\quad &a<0,\; b<0,
\\
\\
2\pi \frac{e^{-\eta a}}{b-a},\quad &a>0,\; b<0,
\\
\\
2\pi \frac{e^{-\eta b}}{a-b},\quad &a<0,\; b>0,
\\
\\
2\pi \left(\frac{e^{-\eta a}}{b-a}+\frac{e^{-\eta b}}{a-b}\right),\quad &a>0,\; b>0.
\end{array}
\right.
\] 
Thus,
\begin{equation}\label{ap1}
\left|\int_{\mathbb{R}}\frac{e^{i\eta \xi}}{(i\xi+a)(i\xi+b)}d\xi\right|\le \frac{2\pi}{|b-a|}\left(e^{-|\eta| |a|}+e^{-|\eta| |b|}\right),\quad \eta \ne 0.
\end{equation}
We verify that \eqref{ap1} still holds when $\eta<0$.

Using again the residue theorem, we can argue as above to obtain
\begin{equation}\label{ap2}
\left|\int_{\mathbb{R}}\frac{e^{i\eta \xi}}{(i\xi+a)^2}d\xi\right|\le 2\pi |\eta| e^{-|\eta| |a|},\quad \eta \ne 0.
\end{equation}

In the case $\mu\ne 0$ or $j\ge 1$, applying \eqref{ap1} to $a=-\tau-\lambda_j^\mu$ and $b=-\tau+\lambda_j^\mu$, we obtain, in view of the fact that $\min(|a|,|b|)=|-|\tau|+\lambda_j^\mu|$,
\begin{equation}\label{ap3}
\left|\int_{\mathbb{R}}\frac{e^{i\eta \xi}}{(i\xi-\tau-\lambda_j^\mu)(i\xi-\tau+\lambda_j^\mu)}d\xi\right|\le \frac{2\pi}{\lambda_j^\mu}e^{-|\eta| ||\tau|-\lambda_j^\mu|}.
\end{equation}
When $\mu=0$ and $j=0$, we get from \eqref{ap2}
\begin{equation}\label{ap4}
\left|\int_{\mathbb{R}}\frac{e^{i\eta \xi}}{(i\xi-\tau-\lambda_j^\mu)(i\xi-\tau+\lambda_j^\mu)}d\xi\right|\le 2\pi |\eta| e^{-|\eta| |\tau|}.
\end{equation}


\begin{thebibliography}{99}
 

\bibitem{Ch}  Choulli, M. : An introduction to unique continuation for second order partial differential equations. Lecture notes (to appear).
 
 \bibitem{DKS} Dos Santos Ferreira, D. ; Kenig, C. E.; Salo, M. : Determining an unbounded potential from Cauchy data in admissible geometries. Commun. Part. Diff. Equat.  38 (1) (2013), 50-68.
 
 \bibitem{JK} Jerison, D. ; Kenig, C. E. : Unique continuation and absence of positive eigenvalues for Schrödinger operators. 
Ann. Math. (2) 121 (1985), no. 3, 463-494.

\bibitem{SU} Bando, S. ; Urakawa, H. :
Generic properties of the eigenvalue of the Laplacian for compact Riemannian manifolds. 
Tôhoku Math. J., II. Ser. 35 (1983), 155-172. 

\bibitem{Sogge} Sogge, C. D. : Fourier integrals in classical analysis. Cambridge Tracts in Math., 210
Cambridge University Press, Cambridge, 2017, xiv+334 pp.

\end{thebibliography}
\end{document}